\documentclass[12pt]{amsart}
\usepackage{graphicx}
\usepackage{amsthm,mathrsfs}
\usepackage{amsfonts}
\usepackage{amsmath}
\usepackage{amssymb}
\newtheorem{lemma}{Lemma}
\newtheorem{theorem}{Theorem}
\newtheorem{corollary}{Corollary}

\newtheorem{remark}{Remark}

\usepackage{float}
\usepackage{url}
\usepackage{epstopdf}
\usepackage{enumerate}

\usepackage[usenames]{color}

\definecolor{darkred}{RGB}{139,0,0}
\definecolor{darkgreen}{RGB}{0,100,0}
\definecolor{darkmagenta}{RGB}{139,0,139}

\newcommand {\ve} {\varepsilon}

\newcommand {\rd} {\,{\rm d}}

\makeatletter\def\blfootnote{\xdef\@thefnmark{}\@footnotetext}\makeatother
\allowdisplaybreaks
\title[Weyl products and uniform distribution]{On Weyl products and uniform distribution modulo one}

\author[Aistleitner]{Christoph Aistleitner}
\address{Institute for Analysis and Number Theory, Graz University of Technology}
\email{aistleitner@math.tugraz.at}

\author[Larcher]{Gerhard Larcher} 
\address{Department of Financial Mathematics and Applied Number Theory, Johannes Kepler University Linz}
\email{gerhard.larcher@jku.at}

\author[Pillichshammer]{Friedrich Pillichshammer} 
\address{Department of Financial Mathematics and Applied Number Theory, Johannes Kepler University Linz}
\email{friedrich.pillichshammer@jku.at}

\author[Saad Eddin]{Sumaia Saad Eddin} 
\address{Department of Financial Mathematics and Applied Number Theory, Johannes Kepler University Linz}
\email{Sumaia.Saad\_Eddin@jku.at}

\author[Tichy]{Robert F. Tichy}
\address{Institute for Analysis and Number Theory, Graz University of Technology}
\email{tichy@tugraz.at}

 \dedicatory{Dedicated to the memory of Edmund Hlawka on the occasion of his hundredth birthday, and to the memory of Hermann Weyl on the centennial of the publication of his fundamental paper.}

\thanks{The authors are supported by the Austrian Science Fund (FWF), Projects F5507-N26 (Larcher and Saad Eddin), F5509-N26 (Pillichshammer) and F5510-N26 (Tichy). These projects are part of the Special Research Program ``Quasi-Monte Carlo Methods: Theory and Applications''. The first author is supported by FWF START-project Y-901.}



\begin{document}

\begin{abstract}
In the present paper we study the asymptotic behavior of trigonometric products of the form $\prod_{k=1}^N 2 \sin(\pi x_k)$ for $N \to \infty$, where the numbers $\omega=(x_k)_{k=1}^N$ are evenly distributed in the unit interval $[0,1]$. The main result are matching lower and upper bounds for such products in terms of the star-discrepancy of the underlying points $\omega$, thereby improving earlier results obtained by Hlawka in 1969. Furthermore, we consider the special cases when the points $\omega$ are the initial segment of a Kronecker or van der Corput sequences The paper concludes with some probabilistic analogues.
\end{abstract}

\date{}
\maketitle

\section{Introduction and statement of the results} \label{sect_1}

Let $f$ be a function $f:[0,1] \mapsto \mathbb{R}_{0}^{+}$ and $(x_{k})_{k \geq 1}$ be a sequence of numbers in the unit interval. Much work was done on analyzing so-called {\em Weyl sums} of the form $S_{N} := \sum^{N}_{k=1} f(x_{k})$, and on the convergence behavior of $\frac{1}{N} S_{N}$ to $\int^{1}_{0} f(x) \rd x$. See for example \cite{b,c,a,d}. It is the aim of this paper to propagate the analysis of corresponding ``{\em Weyl products}''
$$
P_{N} := \prod^{N}_{k=1} f(x_{k}),
$$
in particular with respect to their asymptotic behavior for $N \rightarrow \infty$.

Note that, formally, studying products $P_{N}$ in fact is just a special case of studying $S_{N}$, since 
$$
\log P_{N} = \sum^{N}_{k=1} \log f(x_k),
$$
unless $f(x)= 0$ for some $x \in [0,1]$. Thus we will concentrate on functions $f$ for which $f(0) =0$ (and possibly also $f(1) = 0)$.\\

Assuming an even distribution of the sequence $(x_{k})_{k \geq 1}$, one expects $\frac{1}{N} \sum^{N}_{k=1} \log f(x_k)$ to tend to the integral $\int^{1}_{0} \log f(x) \rd x$ if this exists. That means, very roughly, that we expect
$$
\prod^{N}_{k=1} f(x_{k}) \approx \left({\rm e}^{\int^{1}_{0} \log f(x) \rd x}\right)^{N},
$$
which we can rewrite as
$$
\prod^{N}_{k=1} S_{f} \, f(x_k) \approx 1, \qquad \text{where} \quad S_{f} := {\rm e}^{- \int^{1}_{0} \log f(x) \rd x}.
$$
Hence it makes sense to study the asymptotic behavior of the normalized product 
$$
\prod^{N}_{k=1} S_{f}  f(x_k) \ ~\mbox{ rather than } \ ~\prod^{N}_{k=1} f(x_k).
$$
A special example of such products played an important role in \cite{A-H-L} in the context of pseudorandomness properties of the Thue--Morse sequence, where \emph{lacunary} trigonometric products of the form
$$
\prod^{N}_{k=1} 2 \sin(\pi 2^{k} \alpha)
$$
for $\alpha \in \mathbb{R}$ were analyzed. It was shown there that for almost all $\alpha$ and all $\varepsilon >0$ we have
\begin{equation} \label{tm1}
\prod^{N}_{k=1} |2 \sin( \pi 2^{k} \alpha)| \leq \exp\left((\pi + \varepsilon ) \sqrt{N \log \log N}\right)
\end{equation}
for all sufficiently large $N$ and
\begin{equation} \label{tm2}
\prod^{N}_{k=1} |2 \sin( \pi 2^{k} \alpha) | \geq \exp  \left((\pi-\varepsilon)\sqrt{N \log \log N}\right)
\end{equation}
for infinitely many $N$.\\

In the present paper we restrict ourselves to $f(x) = \sin(\pi x)$ and we will extend the analysis of such products to other types of sequences $(x_k)_{k\geq1}$. In particular we will consider two well-known types of uniformly distributed sequences, namely the van der Corput sequence $(x_k)_{k \geq 1}$ and the Kronecker sequence $(\{k \alpha\})_{k \geq 1}$ with irrational $\alpha \in [0,1]$. Furthermore, we will determine the typical behavior of 
$$
\prod^{N}_{k=1} 2 \sin(\pi x_{k}),
$$
that is, the almost sure order of this product for ``random'' sequences $\left(x_{k}\right)_{k \geq 1}$ in a suitable probabilistic model.\\

Such sine-products and estimates for such products play an important role in many different fields of mathematics. We just mention a few of them: interpolation theory (see \cite{l,m}), partition theory (see \cite{sud,wright}), Pad\'e approximation (see \cite{lub staff}), KAM theory and $q$-series (see \cite{celso,hidet,knill,knillfolk,kuznet}), analytic continuation of Dirichlet series (see \cite{Knill+Les,wagner}), and many more.\\

All our results use methods from uniform distribution theory and discrepancy theory, so we will introduce some of the basic notions from these subjects. Let $x_1, \dots, x_N$ be numbers in $[0,1]$. Their \emph{star-discrepancy} is defined as
$$
D_{N}^{*}=D_N^*(x_{1}, \ldots, x_{N}) = \sup_{a\in [0,1]} \left|\frac{A_{N}(a)}{N} -a \right|,
$$
where $A_{N} (a) := \# \left\{1 \leq n \leq N \ : \ x_n \in [0,a)\right\}$. An infinite sequence $(x_k)_{k \geq 1}$ in $[0,1]$ is called \emph{uniformly distributed modulo one} (u.d. mod 1) if for all $a \in [0,1]$ we have
$$
\lim_{N \to \infty} \frac{A_N(a)}{N} = a,
$$
or, equivalently,
$$
\lim_{N \to \infty} D_N^* = 0.
$$
For more basic information on uniform distribution theory and discrepancy, we refer to \cite{drmotatichy,e}.\\

Now we come to our new results. First we will give general estimates for products $\prod^{N}_{k=1} 2 \sin(\pi x_{k})$ in terms of the star-discrepancy $D_{N}^{*}$ of $(x_k)_{1 \leq k \leq N}$. A similar result in a weaker form was obtained by Hlawka \cite{l} (see also \cite{m}).

\begin{theorem} \label{th_kh}
Let $\left(x_{k}\right)_{k \geq 1}$ be a sequence of real numbers from $[0,1]$ which is u.d. mod 1. Then for all sufficiently large $N$ we have
\begin{equation} \label{equ_a}
\prod^{N}_{k=1} 2 \sin (\pi x_{k}) \leq \left(\frac{N}{\Delta_{N}}\right)^{2 \Delta_{N}},
\end{equation}
where $\Delta_{N} := ND_{N}^{*}$.
\end{theorem}

Concerning the quality of Theorem~\ref{th_kh}, consider the case when $(x_{k})_{k \geq 1}$ is a low-discrepancy sequence such as the van der Corput sequence (which is treated in Theorem~\ref{th_a} below). Then $\Delta_{N} = \mathcal{O}\left(\log N\right)$, and Theorem~\ref{th_kh} gives
\begin{equation} \label{equ_b}
\prod^{N}_{k=1} 2 \sin (\pi x_{k}) \leq N^{\gamma \log N}
\end{equation}
for some $\gamma \in \mathbb{R}^+$ and all sufficiently large $N$. Stronger asymptotic bounds are provided by Theorem~\ref{th_a} below; thus, Theorem \ref{th_kh} does not provide a sharp upper bound in this case.\\

As another example, let $x_k=k/(N+1)$ for $k=1,2,\ldots,N$. This point set has star-discrepancy $D_N^*=1/(N+1)$, and hence the general estimate \eqref{equ_a} gives 
\begin{equation} \label{lhp}
\prod^{N}_{k=1} 2 \sin\left(\pi \frac{k}{N+1}\right) \leq (N+1)^2.
\end{equation}
On the other hand, the product on the left-hand side of \eqref{lhp} is well known to be exactly $N+1$ (see also Lemma~\ref{lem_c_fritz} below). Thus, the general estimate from Theorem \ref{th_kh} has an additional factor $N$ in comparison with the correct order in this case, which is quite close to optimality.\\

As already mentioned above, Hlawka~\cite{l,m} studied similar questions in connection with interpolation of analytic functions on the complex unit disc. There he considered products of the form
$$
\omega_{N}(z)= \prod^{N}_{k=1}(z-\xi_{k})^{2},
$$
where $\xi_{k}$ are points on the unit circle. The main results in
\cite{l,m} are lower and upper bounds of $|\omega_{N}(z)|$ in
terms of the star-discrepancy $D_{N}^*$ of the sequence
$(\arg\frac{1}{2\pi}\xi_{k}), k=1,\ldots , N.$\footnote{The
second paper was published in a seminar proceedings volume called
``Zahlentheoretische Analysis''. Hlawka introduced this term
for applications of number-theoretic methods in
real or complex analysis. In particular, he often applied uniformly distributed
sequences to give discrete versions of continuous models.} It
should also be mentioned that Wagner \cite{wagner}
proved the general lower bound
$$\sup_{| z|= 1}|\omega_{N}(z)|\geq(\log N)^{c}$$
for infinitely $N$, where $c> 0$ is some explicitly given
constant. This solved a problem stated by Erd\H{o}s.\\

In the sequel we will give a second, essentially optimal theorem which estimates products $\prod^{N}_{k=1} 2 \sin (\pi x_{k})$ in terms of the star-discrepancy of the sequence $(x_{k})_{k \geq 1}$. Let $\omega=\left\{x_{1}, \ldots, x_{N}\right\}$ be numbers in $[0,1]$ and let $P_N(\omega)=\prod^{N}_{k=1} 2 \sin(\pi x_k)$. Let $D_{N}^{*} (\omega)$ denote the star-discrepancy of $\omega$. Furthermore, let $d_N$ be a real number from the interval $[1/(2N),1]$, which is the possible range of the star-discrepancy of $N$-element point sets. We are interested in
$$
P_N^{(d_{N})} := \sup_{\omega} P_N(\omega)= \sup_{\omega} \prod^{N}_{k=1}2 \sin(\pi x_k),
$$
where the supremum is taken over all $\omega$ with $D_{N}^{*} (\omega) \leq d_{N}$. We will show
\begin{theorem} \label{th_a_fritz}
Let $\left(d_{N}\right)_{N \geq 1}$ be an arbitrary sequence of reals satisfying $1/(2N) \le d_N \le 1,~N \geq 1,$ and $\lim_{N \rightarrow \infty} d_N = 0$. Then we have:
\begin{itemize}
\item [a)] For all $\varepsilon > 0$ there exist $c(\varepsilon)$ and $N(\varepsilon)$ such that for all $N > N(\varepsilon)$ we  have
$$
P_N^{(d_{N})} \leq c(\varepsilon) \frac{1}{N} \left(\left(\frac{{\rm e}}{\pi}+\varepsilon\right) \frac{1}{d_{N}}\right)^{2N d_{N}}.
$$
\item [b)] For all sufficiently large $N$ we have
$$
P_N^{(d_{N})} \geq \frac{2 \pi^2}{{\rm e}^6} \frac{1}{N}  \left(\frac{{\rm e}}{\pi} \frac{1}{d_N}\right)^{2 N d_N}.
$$
\end{itemize}
\end{theorem}

To check the quality of Theorem \ref{th_a_fritz}, consider the case $d_N=1/(N+1)$ which includes the point sets $x_k=k/(N+1)$ for $k=1,2,\ldots,N$ mentioned before. Then the upper estimate in Theorem~\ref{th_a_fritz} gives the correct order of magnitude $P_N =\mathcal{O}(N)$.\\

Let us now focus on products of the form 
$$
\prod^{N}_{n=1} 2 \sin (\pi \{n \alpha\}) =\prod^{N}_{n=1} 2 \sin (\pi n \alpha),
$$ 
where $\alpha$ is a given irrational number, i.e., we consider the special case when $(x_n)_{n \geq 1}$ is the Kronecker sequence $(\{n \alpha \})_{n \geq 1}$.
Such products play an essential role in many fields and are the best studied such Weyl products in the literature. See for example \cite{h,f,j,i,Knill+Les,Lub,g,VerMes}. Before discussing these products in detail, let us recall some historical facts. By Kronecker's
approximation theorem, the sequence $(n\alpha)_{n \ge 1}$ is everywhere dense modulo 1; i.e., the sequence of fractional
parts $(\{n\alpha\})_{n \geq 1}$ is dense in $[0,1]$. At the
beginning of the 20th century various authors considered this
sequence (and generalizations such as $(\{\alpha n^{d}\})_{n \geq 1}$, etc.) from
different points of view; see for instance Bohl \cite{bohl}, Weyl \cite{weyl2} and Sierpi\'nksi \cite{sierp}. An important impetus came from celestial
mechanics. It was Hermann Weyl in his seminal paper \cite{weyl} who opened new and much more general features of this
subject by introducing the concept of uniform distribution for arbitrary sequences $(x_{k})_{k\geq 1}$ in the unit interval
(as well as in the unit cube $[0,1]^{s}$). This paper heavily influenced the
development of uniform distribution theory, discrepancy theory and the theory of quasi-Monte Carlo integration throughout the last 100
years. For the early history of the
subject we refer to Hlawka and Binder \cite{hla2}.\\

Numerical experiments suggest that for integers $N$ with $q_{l} \leq N < q_{l+1}$, where $\left(q_{l}\right)_{l \geq 0}$ is the sequence of best approximation denominators of $\alpha$,
\begin{equation} \label{equ_stern_fritz}
\text{the product attains its maximal value for}~N = q_{l+1}-1.
\end{equation} 
Moreover we conjecture that always
\begin{equation} \label{equ_6neu_fritz}
\limsup_{q \rightarrow \infty} \frac{1}{q} \prod^{q-1}_{n=1} |2 \sin (\pi n \alpha) | < \infty.
\end{equation}

Compare these considerations also with the conjectures stated in \cite{Lub}.
To illustrate these two assertions see Figures~\ref{fig_b} and \ref{fig_c}, where for  $\alpha = \sqrt{2}$ we plot $\prod^{N}_{n=1} |2 \sin (\pi n \alpha)|$ for $N=1,\ldots, 500$ (Figure~\ref{fig_b}) and the normalized version $\tfrac{1}{N} \prod^{N}_{n=1} |2 \sin (\pi n \alpha) |$ for $N=1, \ldots, 500$ (Figure~\ref{fig_c}). Note that the first best approximation denominators of $\sqrt{2}$ are given by $1,2,5,12,29,70,169,408,\ldots.$\\
\begin{figure}
\begin{center}
\includegraphics[angle=0,width=100mm]{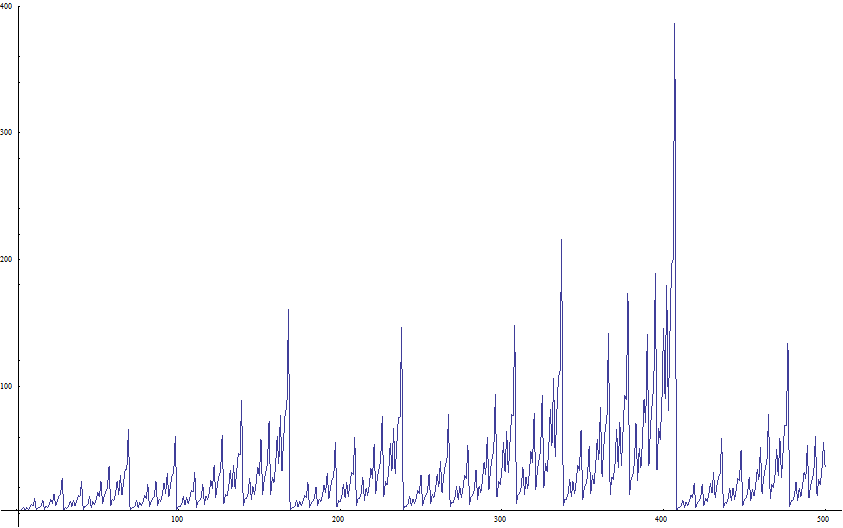}
\caption{$\prod^{N}_{n=1} \left|2 \sin (\pi n \alpha) \right|$ for $N=1,\ldots, 500$ and  $\alpha=\sqrt{2}$ }\label{fig_b}
\end{center}
\end{figure}

\begin{figure}
\begin{center}
\includegraphics[angle=0,width=100mm]{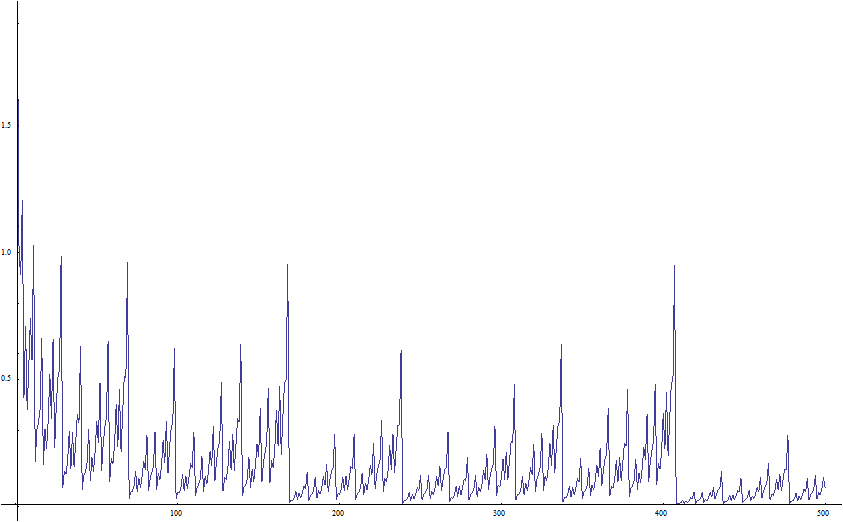}\\
\caption{$\tfrac{1}{N}\prod^{N}_{n=1} \left|2 \sin (\pi n \alpha) \right|$ for $N=1,\ldots, 500$ and  $\alpha=\sqrt{2}$ }\label{fig_c}
\end{center}
\end{figure}

For the case $N= q-1$ for some best approximation denominator $q$ the product $\prod^{q-1}_{n=1} |2 \sin (\pi n \alpha)|$ already was considered in \cite{f,g}. In particular, it was shown there that
\begin{equation} \label{equ_delta_fritz}
\lim_{q \rightarrow \infty}\log \prod^{q-1}_{n=1} |2 \sin (\pi n \alpha)| = \lim_{q \rightarrow \infty} \frac{1}{q} \sum^{q-1}_{n=1} \log|2 \sin (\pi n \alpha)| = 0,
\end{equation}
when $q$ runs through the sequence of best approximation denominators. Indeed, we are neither able to prove assertion \eqref{equ_stern_fritz} nor assertion \eqref{equ_6neu_fritz}. Nevertheless we want to give a quantitative estimate for the case $N=q-1$, i.e., also a quantitative version of \eqref{equ_delta_fritz}, before we will deal with the general case.

\begin{theorem} \label{th_proda}
Let $q$ be a best approximation denominator for $\alpha$. Then
$$
1 \leq \prod^{q-1}_{n=1} |2 \sin (\pi n \alpha)| \leq \frac{q^{2}}{2}.
$$
\end{theorem}


Next we consider general $N\in \mathbb{N}$:

\begin{theorem} \label{th_prodb}
Let $\alpha := [0; a_1, a_2,a_3, \ldots]$ be the continued fraction expansion of the irrational number $\alpha \in [0,1]$. Let $N \in \mathbb{N}$ be given, and denote its Ostrowski expansion by 
$$
N = b_{l}q_{l} + b_{l-1} q_{l-1} + \cdots + b_{1}q_{1} + b_{0}
$$ 
where $l=l(N)$ is the unique integer such that $q_l \le N < q_{l+1}$, where $b_{i} \in \{0,1,\ldots,a_{i+1}\}$, and where $q_1,q_2,\ldots$ are the best approximation denominators for $\alpha$. Then we have
$$
\prod^{N}_{n=1} \left|2 \sin (\pi n \alpha) \right| \leq \prod^{l}_{i=0} 2^{b_{i}} q^{3}_{i}.
$$
\end{theorem}

\begin{corollary} \label{co_prodb}
For all $N$ with $q_{l} \leq N < q_{l+1}$ we have
$$
\frac{1}{N} \sum^{N}_{n=1} \log |2 \sin (\pi n \alpha)| \leq (\log 2) \left(\frac{1}{q_{l}} + \frac{l}{2^{(l-3)/2}}\right) + 3 \, \frac{\log q_{l}}{q_{l}} \left(\frac{\log q_l}{\log \phi}+1 \right),
$$
where $\phi=(1+\sqrt{5})/2$ and hence
$$
\limsup_{N \rightarrow \infty} \frac{1}{N} \sum^{N}_{n=1} \log |2 \sin (\pi n \alpha)| = 0 =\int_0^1 \log(2 \sin(\pi x)) \rd x .
$$
\end{corollary}

The second part of Corollary~\ref{co_prodb} can also be obtained from \cite[Lemma~4]{h}.\\

In the following we say that a real $\alpha$ is of type $ t \ge 1$ if there is a constant $c > 0$ such that $$\left|\alpha - \frac{p}{q}\right| > c  \frac{1}{q^{1+t}}$$ for all $p,q \in \mathbb{Z}$ with $\gcd(p,q)=1$.

The next result essentially improves a result given in \cite{Knill+Les}. There a bound on $\prod^{N}_{n=1} \left|2 \sin \left(\pi n \alpha \right) \right|$ for $\alpha$ of type $t$ of the form $N^{c N^{1-1/t} \log N}$ instead of our much sharper bound $2^{C N^{1-1/t}}$ was given. Note that our result only holds for $t > 1$, so we cannot obtain the sharp result of Lubinsky \cite{Lub} in the case of $\alpha$ with bounded continued fraction coefficients.

\begin{corollary} \label{co_b}
Assume that $\alpha$ is of type $ t > 1$. Then for some constant $C$ and all $N$ large enough $\prod^{N}_{n=1} \left|2 \sin \left(\pi n \alpha \right) \right|\leq 2^{C N^{1-1/t}}$.
\end{corollary}

Now we will deal with $\prod^{N}_{n=1} |2 \sin (\pi x_n)|$, where $(x_n)_{n \geq 1}$ is the van der Corput-sequence. The van der Corput sequence (in base 2) is defined as follows: for $n \in \mathbb{N}$ with binary expansion $n=a_0+a_1 2+a_2 2^3+\cdots$ with digits $a_0,a_1,a_2,\ldots \in \{0,1\}$ (of course the expansion is finite) the $n^{{\rm th}}$ element is given as $$x_n=\frac{a_0}{2}+\frac{a_1}{2^2}+\frac{a_2}{2^3}+\cdots$$ (see the recent survey \cite{FKP} for detailed information about the van der Corput sequence). For this sequence, in contrast to the Kronecker sequence, we can give very precise results. We show:

\begin{theorem} \label{th_a}
Let $(x_{n})_{n \geq 1}$ be the van der Corput sequence in base 2. Then
$$
\limsup_{N \rightarrow \infty} \frac{1}{N^{2}} \prod^{N}_{n=1} |2 \sin (\pi x_{n})| = \frac{1}{2 \pi}
$$
and
$$
\liminf_{N \rightarrow \infty} \prod^{N}_{n=1} |2 \sin (\pi x_{n})| = \pi.
$$
\end{theorem}

Finally, we study probabilistic analogues of Weyl products, in order to be able to quantify the typical order of such products for ``random'' sequences and to have a basis for comparison for the results obtained for deterministic sequences in Theorems \ref{th_proda}, \ref{th_prodb} and \ref{th_a}. We will consider two probabilistic models. First we study $$\prod^{N}_{k=1}2\sin(\pi X_{k}),$$ where $(X_{k})_{k \ge 1}$ is a sequence of independent, identically distributed (i.i.d.) random variables in $[0,1]$. The second probabilistic model are random subsequences $(n_{k}\alpha)_{k \ge 1}$ of the Kronecker sequences $(n\alpha)$, where the elements of
$n_{k}$ are selected from $\mathbb{N}$ independently and with probability $\frac{1}{2}$ for each number. This model is frequently used in the theory of random series
(see for example the monograph of Kahane \cite{kah}) and was introduced to the theory of uniform
distribution by Petersen and McGregor \cite{peter} and later extensively studied by Tichy \cite{tichy}, Losert \cite{losert}, and Losert and Tichy \cite{lt}.

\begin{theorem} \label{th51}
Let $(X_{k})_{k \ge 1}$ be a sequence of i.i.d random variables having uniform distribution on $[0,1]$, and let $$P_{N}= \prod^{N}_{k=1}2\sin (\pi X_{k}).$$ 
Then for all $\varepsilon > 0$ we have, almost surely,
$$
P_N \leq \exp\left(\left(\frac{\pi}{\sqrt{6}} + \varepsilon\right) \sqrt{N \log \log N}\right)
$$
for all sufficiently large $N$, and 
$$
P_N \geq \exp  \left(\left(\frac{\pi}{\sqrt{6}}-\varepsilon \right) \sqrt{N \log \log N}\right)
$$
for infinitely many $N$.\\

\end{theorem}

\begin{theorem} \label{th52}
Let $\alpha$ be an irrational number with bounded continued fraction coefficients. Let $(\xi_n)_{n \geq 1}=(\xi_n(\omega))_{n \geq 1}$ be a sequence of i.i.d.\ $\{0,1\}$-valued random variables with mean $1/2$, defined on some probability space $(\Omega,\mathcal{A},\mathbb{P})$, which induce a random sequence $(n_k)_{k \geq 1}=(n_k(\omega))_{k \geq 1}$ as the sequence of all numbers 
$\left\{n \geq 1:~ \xi_n = 1 \right\}$, sorted in increasing order. Set 
$$
P_N =  \prod^{N}_{k=1}2\sin (\pi n_k \alpha).
$$
Then for all $\ve > 0$ we have, $\mathbb{P}$-almost surely, 
$$
P_N \leq \exp\left(\left(\frac{\pi}{\sqrt{12}} + \varepsilon\right) \sqrt{N \log \log N}\right)
$$
for all sufficiently large $N$, and 
$$
P_N \geq \exp  \left(\left(\frac{\pi}{\sqrt{12}}-\varepsilon \right) \sqrt{N \log \log N}\right)
$$
for infinitely many $N$
\end{theorem}

\begin{remark}\rm
The conclusion of Theorem \ref{th52} remains valid if $\alpha$ is only assumed to be of finite approximation type (see \cite[Chapter 2, Section 3]{e} for details on this notion).  
\end{remark}

\begin{remark}\rm
It is interesting to compare the conclusions of Theorems \ref{th51} (for purely random sequences) and \ref{th52} (for randomized subsequences of linear sequences) to the results in equations \eqref{tm1} and \eqref{tm2}, which hold for lacunary trigonometric products. The results coincide almost excactly, except for the constants in the exponential term (which can be seen as the standard deviations in a related random system; see the proofs). The larger constant in the lacunary setting comes from an interference phenomenon, which appears frequently in the theory of lacunary functions systems (see for example Kac \cite{kac} and Maruyama \cite{maruyama}). On the other hand, the smaller constant in Theorem \ref{th52} represents a ``loss of mass'' phenomenon, which can be observed in the theory of slowly growing (randomized) trigonometric systems; it appears in a very similar form for example in Berkes \cite{berkes} and Bobkov--G\"otze \cite{bobg}.
\end{remark}

The outline of the remaining part of this paper is as follows. In Section \ref{sect_b_end} we will prove Theorems \ref{th_kh} and \ref{th_a_fritz}, which give estimates of Weyl products in terms of the discrepancy of the numbers $(x_k)_{1 \leq k \leq N}$. In Section \ref{sect_c} we prove the results for Kronecker sequences (Theorems \ref{th_proda} and \ref{th_prodb}), and in Section \ref{sect_d} the results for the van der Corput sequence (Theorem \ref{th_a}). Finally, in Section \ref{sect_prob} we prove the results about probabilistic sequences (Theorems \ref{th51} and \ref{th52}). 

\section{Proofs of Theorems \ref{th_kh} and \ref{th_a_fritz}} \label{sect_b_end}

\begin{proof}[Proof of Theorem \ref{th_kh}]
The Koksma-Hlawka-inequality (see e.g. \cite{e}) states that for any function $g:[0,1] \rightarrow \mathbb{R}$ of bounded variation $V(g)$, any $N$ and numbers $x_1, \dots, x_N \in [0,1]$ we have
$$
\left|\int^{1}_{0} g(x) \rd x - \frac{1}{N} \sum^{N}_{k=1} g\left(x_{k}\right)\right| \leq V(g) \, D^{*}_{N} (x_1, \dots, x_N),
$$
where $D_{N}^{*}$ is the star-discrepancy of $x_{1}, \ldots, x_{N}$. Let $P_{N} := \prod^{N}_{k=1} 2 \sin( \pi x_{k})$ and $$\Sigma_{N} := \log P_{N} = N \log 2 + \sum^{N}_{k=1} \log \sin(\pi x_{k}).$$ For $0 < \varepsilon < \frac{1}{2}$ let
\begin{align*}
f_{\varepsilon} (x) := \left\{ \begin{array}{ll}
\log \sin (\pi \varepsilon) & \mbox{if}~\left\|x\right\| \leq \varepsilon \\
\log \sin (\pi x) & \mbox{otherwise.}\\
\end{array} \right.
\end{align*}
Note, that $\int^{1}_{0} \log \sin (\pi x) \rd x = - \log 2$, hence 
\begin{align*}
\int^{1}_{0} f_{\varepsilon} (x) ~\rd x  = & 2 \varepsilon \log \sin (\pi \varepsilon) + \int^{1}_{0} \log \sin (\pi x) \rd x - 2 \int^{\varepsilon}_{0} \log \sin (\pi x) \rd x \\
 = & 2 \varepsilon \log \sin (\pi \varepsilon) - \log 2 - 2 \int^{\varepsilon}_{0} \log \sin (\pi x) \rd x.
\end{align*}
By partial integration we obtain 
\begin{align*}
\int^{\varepsilon}_{0} \log \sin (\pi x) \rd x  = & \varepsilon \log \sin (\pi \varepsilon) - \int^{\varepsilon}_{0} x \pi \cot (\pi x) \rd x \\
 = & \varepsilon \log \sin (\pi \varepsilon) - \varepsilon - \mathcal{O}(\varepsilon^{3}) 
\end{align*}
(with a positive $\mathcal{O}$-constant for $\varepsilon$ small enough). 
Furthermore, we have
\begin{equation*}
V(f_{\varepsilon}) = \int_0^1 |f_{\varepsilon}'(x)| \rd x = 2 \pi \int_{\varepsilon}^{1/2} \cot(\pi x) \rd x = -2 \log \sin(\pi \varepsilon). 
\end{equation*}

Altogether we have, using the Koksma-Hlawka inequality and since $\log \sin (\pi \varepsilon) = \log (\pi \varepsilon) - \frac{\pi \varepsilon^{2}}{6} - \mathcal{O}(\varepsilon^{4}),$
\begin{align*}
\Sigma_{N}  \leq & N \log 2 + \sum^{N}_{k=1} f_{\varepsilon} \left(x_{k}\right)\\
 \leq & N \log 2 + N \int^{1}_{0} f_{\varepsilon} (x) \rd x + N D_{N}^{*} V(f_{\varepsilon})\\
 =& N\left(2 \varepsilon \log \sin \varepsilon - 2 \int^{\varepsilon}_{0} \log \sin (\pi x) \rd x \right) - 2 N D_{N}^{*} \log \sin (\pi \varepsilon) \\
 = & 2N \int^{\varepsilon}_{0} x \pi \cot (\pi x) \rd x - 2ND_{N}^{*} \log \sin (\pi \varepsilon) \\
 = & 2 N \varepsilon + N \mathcal{O}(\varepsilon^{3}) + 2 N D_{N}^{*} \, (-\log (\pi \varepsilon) + \mathcal{O} (\varepsilon^{2})) \\
 = & 2 N \varepsilon - 2 N D_{N}^{*} \log \pi \varepsilon + N \mathcal{O}(\varepsilon^{2}).
\end{align*}
Hence
$$
P_{N}={\rm e}^{\Sigma_N} \leq {\rm e}^{2 N \varepsilon} \left(\frac{1}{\pi \varepsilon}\right)^{2 N D_{N}^{*}} {\rm e}^{c  \varepsilon^{2} N}
$$
for some constant $c>0$. We choose $\varepsilon = D_{N}^{*}$ and obtain
$$
P_{N} \leq \left(c' \, \frac{N}{N D_{N}^{*}}\right)^{2 ND_{N}^{*}}
$$
For some $c'>0$. Note that $c'$ can be chosen such that $c' < 1$ if $\varepsilon = D_{N}^{*} =o(1)$ for $N \rightarrow \infty$. 
\end{proof}

Next we come to the proof of Theorem \ref{th_a_fritz}. We will need several auxiliary lemmas, before proving the theorem.

\begin{lemma} \label{lem_a_fritz}
Let $D \in \mathbb{Q}$, let $N$ be even such that $N D =: M$ for some integer $M$. Then the $N$-element point set $\widetilde{\omega}$ consisting of the points $$\frac{M}{N}, \frac{M+1}{N}, \ldots, \frac{\frac{N}{2}-1}{N}, \frac{\frac{N}{2}+1}{N}, \frac{\frac{N}{2}+2}{N}, \ldots, \frac{N-M}{N}$$ together with $2 M$ times the point $\frac{1}{2}$ has star-discrepancy
\begin{equation*} 
D_{N}^{*} (\omega) = D.
\end{equation*}
If any of these points is moved nearer to $1/2$, then the star-discrepancy of the new point set is larger than $D$. Furthermore, $\widetilde{\omega}$ is the only sequence with these two properties.
\end{lemma}

\begin{proof}
See Figure~\ref{bd78}, where the discrepancy function $a \mapsto \frac{A_{N}(a)}{N} -a$ for $a \in [0,1]$ of $\widetilde{\omega}$ is plotted.
\begin{figure}[ht]
\begin{center}
\input{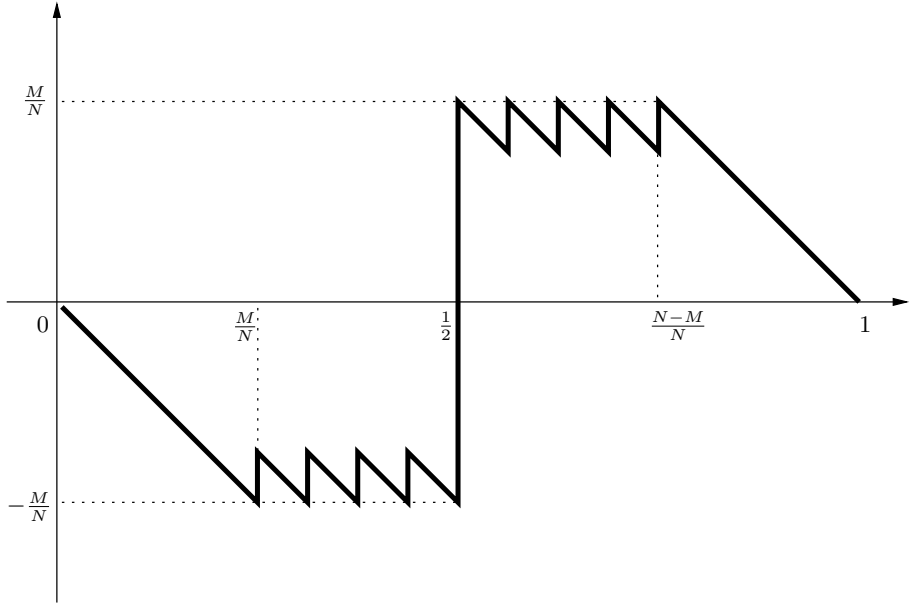}
\caption{Discrepancy function $a \mapsto \frac{A_{N}(a)}{N} -a$ of $\widetilde{\omega}$} \label{bd78}
\end{center}
\end{figure}
\end{proof}

\begin{lemma} \label{lem_b_fritz}
For $\widetilde{\omega}$ as in Lemma~\ref{lem_a_fritz} we have $P_N^{(d_{N})} = P_N(\widetilde{\omega})$.
\end{lemma}

\begin{proof}
Let $x_{1}, \ldots, x_{N}$ be any $N$-element point set in $[0,1]$. If one of these points is moved nearer to $1/2$ then this move increases the value $\prod^{N}_{k=1} 2 \sin (\pi x_{k})$. Hence the result immediately follows from Lemma~\ref{lem_a_fritz}.
\end{proof}

\begin{lemma} \label{lem_c_fritz}
For all $N\in \mathbb{N}$ and all $x\in [0,1]$ we have 
\begin{enumerate}[(i)]
\item \quad $\prod^{N-1}_{k=1} 2 \sin(\pi k/N) = N$, and
\item \quad $\prod^{N-1}_{k=0} 2 \sin(\pi (k+x)/N) = 2 \sin(\pi x).$
\end{enumerate}
\end{lemma}

\begin{proof}
Equation (i) is well known. A nice proof can be found for example in \cite{mu}. Equation (ii) is \cite[Formula 1.392]{GR}. 
\end{proof}

\begin{lemma} \label{lem_d_fritz}
There is an $\varepsilon_{0} > 0$ such that for all $\varepsilon < \varepsilon_{0}$ we have
$$
\varepsilon  \log (\pi \varepsilon) - \varepsilon - \varepsilon^{2} \leq \int^{\varepsilon}_{0} \log \sin(\pi x) \rd x \leq \varepsilon \log (\pi \varepsilon) - \varepsilon.
$$
\end{lemma}

\begin{proof}
This follows immediately from the Taylor expansion
$$
\int^{\varepsilon}_{0} \log \sin(\pi x) \rd x - \varepsilon  \log (\pi \varepsilon) = - \varepsilon - \frac{\pi^{2}}{18} \varepsilon^{3} + \mathcal{O}(\varepsilon^{5}).
$$
\end{proof}

\begin{lemma} \label{lem_e_fritz}
There is an $\varepsilon_{0} > 0$ such that for all $\varepsilon < \varepsilon_{0}$ we have
$$
\log (\pi \varepsilon) - \varepsilon \leq \log \sin (\pi \varepsilon) \leq \log (\pi \varepsilon).
$$
\end{lemma}

\begin{proof}
This follows from
$$
\log \sin (\pi x) - \log (\pi x) = - \frac{\pi^{2}x^{2}}{6} + \mathcal{O}(x^{4}).
$$
\end{proof}

\begin{proof}[Proof of Theorem~\ref{th_a_fritz}]
Let $N d_N = M$ with $M \geq 2$ (for $M=1$ the result is easily checked by following the considerations below) and $\widetilde{\omega}$ as in Lemmas~\ref{lem_a_fritz} and \ref{lem_b_fritz}. Note that $M=M(N)$ depends on $N$. We have, using also equation (i) of Lemma~\ref{lem_c_fritz},
\begin{align*}
P_N(\widetilde{\omega}) 
 = & \left(\prod^{N-1}_{k=1} 2 \sin \left(\pi \frac{k}{N}\right)\right) \ 2^{2 M-1} \ \left(\prod^{M-1}_{k=1} 2 \sin \left(\pi \frac{k}{N}\right)\right)^{-2}\\
 = & 2 N\left(\prod^{M-1}_{k=1} \sin \left(\pi \frac{k}{N}\right)\right)^{-2}.
\end{align*}
Note that the function $x \mapsto \log \sin (\pi x)$ is of the form as presented in Figure~\ref{fig_a}.
\begin{figure}
\begin{center}
\includegraphics[angle=0,width=100mm]{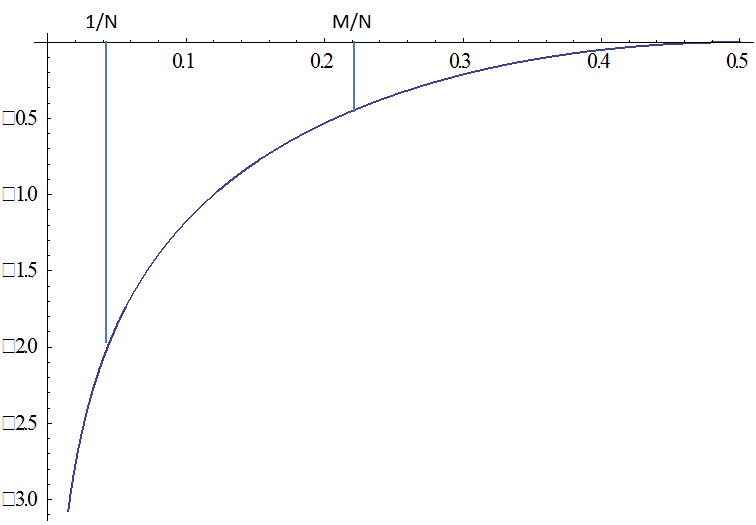}
\caption{The function $\log \sin (\pi x)$}\label{fig_a}
\end{center}
\end{figure}
Hence for $M < \frac{N}{2}$ we have
\begin{align*}
\log \sin \left(\frac{\pi}{N}\right)  + & N \int^{\frac{M-1}{N}}_{\frac{1}{N}} \log \sin (\pi x) \rd x\\
 \leq & \sum^{M-1}_{k=1} \log \sin \left(\pi \frac{k}{N}\right)\\
 \leq & N  \int^{\frac{M-1}{N}}_{\frac{1}{N}} \log \sin (\pi x) \rd x + \log \sin \left(\pi \frac{M-1}{N}\right).
\end{align*}
By Lemma~\ref{lem_d_fritz} for all $M$ with $\frac{M}{N} < \varepsilon_{0}$ for the integral above  we have
\begin{align*}
\lefteqn{N \int^{\frac{M-1}{N}}_{\frac{1}{N}} \log \sin (\pi x) \rd x}\\
&\leq \left(M-1\right) \log \left(\pi \frac{M-1}{N}\right) - (M-1) - \log \left(\frac{\pi}{N}\right) + 1 + \frac{1}{N},
\end{align*}
and hence, using also Lemma~\ref{lem_e_fritz},
\begin{align*}
\sum^{M-1}_{k=1} \log \sin \left(\pi \frac{k}{N} \right)  \leq & (M-1)  \log \left(\frac{\pi}{{\rm e}} \frac{M-1}{N}\right) - \log \pi +\log N + 1\\
&  + \frac{1}{N} + \log (M-1) - \log N + \log \pi \\
 \le & (M-1)  \log \left(\frac{\pi}{{\rm e}} \frac{M-1}{N}\right) + \log (M-1) + 2.
\end{align*}
This gives 
$$
\left(\prod^{M-1}_{k=1} \sin \left(\pi \frac{k}{N}\right)\right)^{2} =  {\rm e}^{2 \sum^{M-1}_{k=1} \log \sin (\pi k/N)} \leq {\rm e}^{4} (M-1)^{2} \left(\frac{\pi}{{\rm e}} \frac{M-1}{N}\right)^{2 (M-1)},
$$
and consequently
\begin{align*} 
P_N(\widetilde{\omega})  \geq & 2N \frac{1}{{\rm e}^{4}} \frac{1}{(M-1)^2}  \left(\frac{{\rm e}}{\pi} \frac{N}{M-1}\right)^{2(M-1)} =  \frac{2 \pi^2}{{\rm e}^6} \frac{1}{N}  \left(\frac{{\rm e}}{\pi} \frac{N}{M-1}\right)^{2 M}\nonumber\\
 \ge & \frac{2 \pi^2}{{\rm e}^6} \frac{1}{N}  \left(\frac{{\rm e}}{\pi} \frac{1}{d_N}\right)^{2 N d_N}.
\end{align*}
This proves assertion b) of Theorem~\ref{th_a_fritz}.\\

On the other hand we have
\begin{align*}
\lefteqn{N  \int^{\frac{M-1}{N}}_{\frac{1}{N}} \log \sin (\pi x) \rd x}\\
&\geq (M-1) \log \left(\pi \frac{M-1}{N}\right) -(M-1) - \frac{(M-1)^{2}}{N} - \log \left(\frac{\pi}{N}\right)+1,
\end{align*}
and hence
\begin{align*}
\sum^{M-1}_{k=1} \log \sin \left(\pi \frac{k}{N}\right)  \geq & (M-1) \log \left(\frac{\pi}{{\rm e}} \frac{M-1}{N}\right) \\
&- \frac{(M-1)^{2}}{N} - \log \pi + \log N + 1 + \log \pi - \log N - \frac{1}{N}\\
 = & (M-1) \log \left(\frac{\pi}{{\rm e}} \frac{M-1}{N}\right) + 1-\frac{1}{N} - \frac{(M-1)^{2}}{N}.
\end{align*}
This gives 
$$
\left(\prod^{M-1}_{k=1} \sin \left(\pi \frac{k}{N}\right)\right)^{2} =  {\rm e}^{2 \sum^{M-1}_{k=1} \log \sin (\pi k/N)} \geq \frac{1}{{\rm e}^{\frac{2 (M-1)^{2}}{N}}}  \left(\frac{\pi}{{\rm e}} \frac{M-1}{N}\right)^{2(M-1)},
$$
and consequently
\begin{equation} \label{equ_d_fritz}
P_N \left(\widetilde{\omega}\right) \leq 2 N {\rm e}^{2 \frac{(M-1)^{2}}{N}}  \left(\frac{{\rm e}}{\pi} \frac{N}{M-1}\right)^{2 (M-1)}.
\end{equation}

It remains to show that for all $\varepsilon > 0$ there are $c(\varepsilon)$ and $N(\varepsilon)$ such that for all $N \geq N(\varepsilon)$ the right hand side of \eqref{equ_d_fritz} is at most $c(\varepsilon) \frac{1}{N} ((\tfrac{{\rm e}}{\pi} + \varepsilon ) \frac{N}{M})^{2 M}$.

To this end let $B(\varepsilon)$ be large enough such that for all $M > B(\varepsilon)$ we have $(M-1)^{1/M}  \frac{M}{M-1} < 1 + \frac{\pi}{2 {\rm e}} \varepsilon$. Furthermore, let $N(\varepsilon)$ be large enough such that for all $N \geq N(\varepsilon)$ the value $\frac{M}{N} = d_{N}$ is so small such that $${\rm e}^{\frac{M-1}{N}} < \frac{1+\frac{\pi}{{\rm e}} \varepsilon}{1 + \frac{\pi}{2 {\rm e}} \varepsilon}.$$ Then for all $M > B(\varepsilon)$ and all $N > N(\varepsilon)$ we have
\begin{align*}
\lefteqn{2 N {\rm e}^{2 \frac{\left(M-1\right)^{2}}{N}} \left(\frac{{\rm e}}{\pi} \frac{N}{M-1}\right)^{2 (M-1)}}\\
& \le  \frac{2 \pi^2}{{\rm e}^2} \frac{\left(M-1\right)^{2}}{N} \left(\frac{{\rm e}}{\pi} \, {\rm e}^{\frac{M-1}{N}} \frac{M}{M-1} \frac{N}{M} \right)^{2 M}\\
& =  \frac{2 \pi^2}{{\rm e}^2 N} \left(\frac{{\rm e}}{\pi} \, {\rm e}^{\frac{M-1}{N}} \, (M-1)^{\frac{1}{M}} \, \frac{M}{M-1} \, \frac{N}{M}\right)^{2 M}\\
& \leq \frac{2 \pi^2}{{\rm e}^2 N} \left(\left(\frac{{\rm e}}{\pi} + \varepsilon \right)  \frac{N}{M}\right)^{2 M}.
\end{align*}
If $M \leq B(\varepsilon)$, then the penultimate expression can be estimated by 
\begin{align*}
\lefteqn{\frac{2 \pi^2}{{\rm e}^2N} \left(\frac{{\rm e}}{\pi} \, {\rm e}^{\frac{M-1}{N}} \, (M-1)^{\frac{1}{M}} \, \frac{M}{M-1} \, \frac{N}{M}\right)^{2 M}}\\
& \leq  \left(\max_{M\leq B(\varepsilon)}\left(\frac{2 \pi^2}{{\rm e}^2} \, {\rm e}^{2 M (M-1)} (M-1)^2 \left(\frac{M}{M-1}\right)^{2 M}\right)\right)  \frac{1}{N} \left(\frac{{\rm e}}{\pi} \frac{N}{M}\right)^{2 M}\\
&=  c(\varepsilon)  \frac{1}{N} \left(\frac{{\rm e}}{\pi} \frac{N}{M}\right)^{2 M},
\end{align*}
where $$c(\varepsilon):=\max_{M\leq B(\varepsilon)}\left(\frac{2 \pi^2}{{\rm e}^2} \, {\rm e}^{2 M \left(M-1\right)} \left(M-1\right)^{2} \left(\frac{M}{M-1}\right)^{2 M}\right).$$
This implies the desired result.
\end{proof}

\section{Proofs of the results for Kronecker sequences} \label{sect_c}

\begin{proof}[Proof of Theorem~\ref{th_proda}]
Let $\alpha = \frac{p}{q} + \theta$ with $0 < \theta < \frac{1}{q \cdot q^{+}}$, where $q^{+}$ is the best approximation denominator following $q$. The case of negative $\theta$ can be handled quite analogously. There is exactly one of the points $\{k \alpha\}$ for $k=1, \dots, q-1$ in each interval $[\frac{m}{q}, \frac{m+1}{q})$ for $m=1, \ldots, q-1$. Note that the point in the interval $[\frac{q-1}{q}, 1)$ is the point $\left\{q^{-} \alpha\right\}$, where $q^{-}$ is the best approximation denominator preceding $q$. We have
$$
\left\{q^{-} \alpha \right\} = \frac{q-1}{q} + q^{-} \theta \leq \frac{q-1}{q} + \frac{q^{-}}{q \cdot q^{+}} < \frac{q-1}{q} + \frac{1}{2 q} = \frac{q- \frac{1}{2}}{q}.
$$
Hence, on the one hand (by equation (i) of Lemma~\ref{lem_c_fritz}),
$$
\prod^{q-1}_{n=1} |2 \sin( \pi n \alpha)| \leq \left(\prod^{q-1}_{n=2} 2 \sin\left( \pi \frac{n}{q} \right)\right) \, 2 \sin\frac{\pi}{2} = \frac{2 q}{2 \sin (\pi/q)} \leq \frac{q^2}{2}.
$$
On the other hand
\begin{align*}
\prod^{q-1}_{n=1} |2 \sin (\pi n \alpha)| \geq & \left(\prod^{q-1}_{n=1} 2 \sin \left(\pi \frac{n}{q}\right)\right) \ \frac{1}{2 \sin(\pi \tfrac{\lfloor q/2\rfloor}{q})} \ 2 \sin (\pi q^{-} \alpha) \\
\geq & q \, \sin\left(\pi \frac{q-1/2}{q}\right) = q \, \sin \frac{\pi}{2 q} \geq 1.
\end{align*}
\end{proof}

\begin{proof}[Proof of Theorem~\ref{th_prodb}]
Let $N_{i} := b_{l}q_{l} + b_{l-1}q_{l-1} + \cdots + b_{i+1} q_{i+1}$ for $i=0, \ldots, l-1$ and $N_{l} := 0$. Then 
$$
\prod^{N}_{n=1} \left|2 \sin( \pi n \alpha) \right| = \prod^{l}_{i=0} \prod^{N_{i} + b_{i}q_{i}}_{n=N_{i}+1} \left|2 \sin (\pi n \alpha) \right|.
$$
We consider 
$$
\Pi_{i} := \prod^{N_{i}+b_{i}q_{i}}_{n=N_{i}+1} \left|2 \sin (\pi n \alpha) \right|.
$$
Let $\alpha := \frac{p_{i}}{q_{i}} + \theta_{i}$ with, say, $\frac{1}{2 q_{i} q_{i+1}} < \theta_{i} < \frac{1}{q_{i} q_{i+1}}$. (The case of negative $\theta_{i}$ is handled quite analogously.)

Let $n=N_{i} + dq_{i} + k$ for some $0 \leq d < b_{i}$ and $1 \leq k \leq q_{i}$, then, with $\kappa:= \kappa_{i} := \{N_{i} \alpha \} \pmod{\frac{1}{q_{i}}}$ and $\tilde{\theta}_i:=q_i \theta_i$
we have
\begin{equation} \label{equ_proda}
\left\{n \alpha \right\} = \left\{N_{i} \alpha + k \frac{p_{i}}{q_{i}} + (dq_{i}+k) \theta_{i} \right\}= \left\{\kappa + \frac{l(k)}{q_{i}} + d \tilde{\theta}_i +k \theta_{i}\right\}
\end{equation}
for some $l(k)\in \{0,1,\ldots, q_i -1\}$. Since $0 < k \theta_{i} + dq_{i} \theta_{i} \leq \frac{a_{i+1}q_{i}}{q_{i+1}q_{i}} < \frac{1}{q_{i}}$, for given $d$ there is always exactly one point $\left\{n \alpha\right\}$ in the interval $[\kappa + \frac{l}{q_{i}} , \kappa + \frac{l+1}{q_{i}}) =: I_{l}$ for each $l=0,\ldots,q_{i}-1$ (the interval taken modulo one).

We replace now the points $\left\{n \alpha\right\}$ by new points, namely: 

\begin{itemize}
\item if $\left\{n \alpha \right\} \in I_{l}$ with $\kappa + \frac{l}{q_{i}} \geq\frac{1}{2}$ then in the representation \eqref{equ_proda} of $\left\{n \alpha \right\}$ we replace $k \theta_{i}$ by $0$, unless $l=q_{i}-1$.

\item if $\left\{n \alpha\right\} \in I_{l}$ with $\kappa + \frac{l+1}{q_{i}} < \frac{1}{2}$ then in the representation \eqref{equ_proda} of $ \left\{n \alpha \right\}$ we replace $k \theta_{i}$ by $\tilde{\theta}_{i}.$

\item if $\left\{ n \alpha \right\} \in I_{l_{0}}$, where $l_{0}$ is such that $\kappa + \frac{l_{0}}{q_{i}} < \frac{1}{2} \leq \kappa + \frac{l_{0}+1}{q_{i}}$ then
\begin{itemize}
\item for the $d$ such that $\kappa + \frac{l_{0}}{q_{i}} + d \tilde{\theta}_{i} \geq \frac{1}{2}$ in the representation \eqref{equ_proda} of $\left\{n \alpha\right\}$ we replace $k \theta_{i}$ by $0$, 
\item for the $d$ such that $\kappa + \frac{l_{0}}{q_{i}} + (d+1) \tilde{\theta}_{i} < \frac{1}{2}$ in the representation \eqref{equ_proda} of $\left\{n \alpha \right\}$ we replace $k \theta_{i}$ by $\tilde{\theta}_{i}$, 
\item for the single $d_{0}$ such that $\kappa + \frac{l_{0}}{q_{i}} + d_{0} \tilde{\theta}_{i} < \frac{1}{2} \leq \kappa + \frac{l_{0}}{q_{i}} + \left(d_{0}+1\right)\tilde{\theta}_i$ we replace $\left\{n \alpha\right\}$ by~$\frac{1}{2}$.
\end{itemize}

\item if $\left\{n \alpha \right\} \in I_{l}$ with $l=q_{i}-1$, then
\begin{itemize}
\item for the $h$ such that $\kappa + \frac{q_{i}-1}{q_{i}} + h \tilde{\theta}_{i} \geq 1$ in the representation \eqref{equ_proda} of $\left\{n\alpha \right\}$ we replace~$k \theta_{i}$~by~$\tilde{\theta}_{i}$,
\item for the $h$ such that
$\kappa + \frac{q_{i}-1}{q_{i}} + (h+1) \tilde{\theta}_{i} \leq 1$ in the representation \eqref{equ_proda} of $\left\{n \alpha \right\}$ we replace $k \theta_{i}$ by $0$, 
\item for the single $h_{0}$ such that $\kappa + \frac{q_{i}-1}{q_{i}} + h_{0} \tilde{\theta}_{i} < 1 < \kappa + \frac{q_{i}-1}{q_{i}} + \left(h_{0}+1\right) \tilde{\theta}_{i}$ we replace in the representation \eqref{equ_proda} of $\left\{ n \alpha \right\}$ the $k \theta_{i}$ by $0$ if $g(\kappa + \frac{q_{i}-1}{q_{i}} + h_{0} \tilde{\theta}_{i}) \geq g(\kappa + \frac{q_{i}-1}{q} + (h_{0}+1) \tilde{\theta}_{i})$ and by $\tilde{\theta}_{i}$ otherwise, where here and in the following we use the notation $g(x) := \left|2 \sin \pi x \right|$. Let the second be the case, the other case is handled quite analogously.
\end{itemize}
\end{itemize}

Using the new points instead of the $\left\{n \alpha\right\}$ by construction we obtain an upper bound $\widetilde{\Pi}_{i}$ for $\Pi_{i}$. 
Then
\begin{align*}
\widetilde{\Pi}_{i} = &  g(\kappa + \tilde{\theta}_{i} ) g(\kappa+ 2 \tilde{\theta}_{i}) \cdots g(\kappa+b_{i} \tilde{\theta}_{i})\\
& \times 
g(\kappa + \tfrac{1}{q_{i}}+ \tilde{\theta}_{i})  g(\kappa+\tfrac{1}{q_{i}} + 2 \tilde{\theta}_{i}) \cdots  g(\kappa + \tfrac{1}{q_{i}}+b_{i}\tilde{\theta}_{i}) \\
& 
\vdots \\
& \times 
g(\kappa + \tfrac{l_{0}-1}{q_{i}} + \tilde{\theta}_{i}) g(\kappa + \tfrac{l_{0}-1}{q_{i}} + 2 \tilde{\theta}_{i}) \cdots g(\kappa + \tfrac{l_{0}-1}{q_{i}} + b_{i} \tilde{\theta}_{i}) \\
& \times g(\kappa+\tfrac{l_{0}}{q_{i}} + \tilde{\theta}_{i}) \cdots g(\kappa + \tfrac{l_{0}}{q_{i}} + d_{0} \tilde{\theta}_{i})  g(\tfrac{1}{2}) \\
& \hspace{1cm}\times g(\kappa + \tfrac{l_{0}}{q_{i}} + (d_{0}+1)\tilde{\theta}_{i}) \cdots  g(\kappa + \tfrac{l_{0}}{q_{i}} + (b_{i}-1)\tilde{\theta}_{i}) \\ 
& \times g(\kappa + \tfrac{l_{0}+1}{q_{i}})  g(\kappa + \tfrac{l_{0}+1}{q_{i}} + \tilde{\theta}_{i}) \cdots g(\kappa + \tfrac{l_{0}+1}{q_{i}} + (b_{i}-1)\tilde{\theta}_{i})\\
& \vdots \\
& \times g(\kappa + \tfrac{q_{i}-2}{q_{i}}) g(\kappa + \tfrac{q_{i}-2}{q_{i}} + \tilde{\theta}_{i}) \cdots g(\kappa + \tfrac{q_{i}-2}{q_{i}} + (b_{i}-1)\tilde{\theta}_{i}) \\
& \times 
g(\kappa + \tfrac{q_{i}-1}{q_{i}}) \cdots g(\kappa + \tfrac{q_{i}-1}{q_{i}} + (h_{0}-1) \tilde{\theta}_{i}) g(\kappa + \tfrac{q_{i}-1}{q_{i}} + (h_{0} + 1) \tilde{\theta}_{i}) \\
& \hspace{1cm}\times 
g(\kappa + \tfrac{q_{i}-1}{q_{i}} + (h_{0}+2)  \tilde{\theta}_{i}) \cdots  g(\kappa + \tfrac{q_{i}-1}{q_{i}} + b_{i}\tilde{\theta}_{i}).
\end{align*}
Hence
\begin{align*}
\widetilde{\Pi}_{i} = &  \left(\prod^{b_{i}-1}_{d=1} \prod^{q_{i}-1}_{l=0} g \left(\kappa + \frac{l}{q_{i}} + d \tilde{\theta}_{i}\right)\right)  \frac{g(\frac{1}{2})}{g(\kappa + \frac{q_{i}-1}{q_{i}} + h_{0} \tilde{\theta}_{i})} \\
& \times \left(\prod^{l_{0}-1}_{l=0} g \left(\kappa + \frac{l}{q_{i}} + b_{i}\tilde{\theta}_{i}\right)\right)  \prod^{q_{i}-1}_{l=l_{0}+1} g \left(\kappa + \frac{l}{q_{i}}\right).
\end{align*}
By equation (ii) of Lemma~\ref{lem_c_fritz} we have
$$
\prod^{q_{i}-1}_{l=0} g \left(\kappa + \frac{l}{q_{i}} + d \tilde{\theta}_{i} \right) = 2 |\sin (\pi q_{i} (\kappa + d \tilde{\theta}_{i}))| \leq 2
$$
and hence $$\prod^{b_{i}-1}_{d=1} \prod^{q_{i}-1}_{l=0} g \left(\kappa + \frac{l}{q_{i}} + d \tilde{\theta}_{i}\right) \le 2^{b_i-1} |\sin (\pi q_{i} (\kappa + h_{0} \tilde{\theta}_{i}))|.$$
Note that $b_{i} \tilde{\theta}_i <  \frac{a_{i+1}}{q_{i+1}} < \frac{1}{q_{i}}$ and therefore also $\kappa + d \tilde{\theta}_{i} < \frac{2}{q_{i}}$ always. Hence
\begin{align*}
\lefteqn{\left(\prod^{l_{0}-1}_{l=0} g \left(\kappa + \frac{l}{q_{i}} + b_{i}\tilde{\theta}_{i}\right)\right) \prod^{q_{i}-1}_{l=l_{0}+1} g \left(\kappa + \frac{l}{q_{i}}\right)}\\
& \leq g \left(\frac{2}{q_{i}}\right)  g \left(\frac{3}{q_{i}}\right) \cdots g \left(\frac{\lfloor q_{i}/2\rfloor}{q_{i}}\right) g \left(\frac{1}{2}\right)^{2}  g \left(\frac{\lfloor q_{i}/2\rfloor+1}{q_{i}}\right)  \cdots g \left(\frac{q_{i}-1}{q_{i}}\right)\\
&= \left(\prod^{q_{i}-1}_{l=1} 2 \sin\left( \pi \frac{l}{q_{i}}\right)\right)  \frac{4}{\sin(\pi/q_{i})}= \frac{4q_{i}}{\sin( \pi/q_{i})} \leq 2 q_{i}^{2}.
\end{align*}
Hence 
\begin{eqnarray*}
\widetilde{\Pi}_{i} \leq 2^{b_{i}-1} \, \frac{2 |\sin (\pi q_{i} (\kappa + h_{0} \tilde{\theta}_{i}))|
}{2 |\sin(\pi (\kappa + \tfrac{q_{i}-1}{q_{i}} + h_{0}\tilde{\theta}_{i}))|} \, 2 q_{i}^{2}.
\end{eqnarray*}
We have
$$\frac{|\sin (\pi q_{i} (\kappa + h_{0} \tilde{\theta}_{i}))|
}{|\sin(\pi (\kappa + \tfrac{q_{i}-1}{q_{i}} + h_{0}\tilde{\theta}_{i}))|}=\frac{|\sin (\pi q_{i} (\kappa +\tfrac{q_i -1}{q_i}+ h_{0} \tilde{\theta}_{i}))|
}{|\sin(\pi (\kappa + \tfrac{q_{i}-1}{q_{i}} + h_{0}\tilde{\theta}_{i}))|} \le q_i,$$ since $|\sin(nx)/\sin x| \le n$ for $n \in \mathbb{N}$. Hence $$\widetilde{\Pi}_{i} \leq  2^{b_{i}} q_{i}^{3}$$ and therefore
$$
\prod^{N}_{n=1} |2 \sin (\pi n \alpha)| \leq \prod^{l}_{i=0} 2^{b_{i}} q_{i}^{3},
$$
as desired.
\end{proof}

\begin{proof}[Proof of Corollary~\ref{co_prodb}]
By Theorem~\ref{th_prodb} we have
\begin{align*}
\frac{1}{N} \sum^{N}_{n=1} \log |2 \sin (\pi n \alpha)| \leq & (\log 2) \frac{b_{0} + \cdots + b_{l}}{b_{0} q_{0} + \cdots + b_{l} q_{l}} +3  \frac{\log q_{1} + \cdots + \log q_{l}}{b_{0} + b_{1}q_{1} + \cdots + b_{l}q_{l}} \\
\leq & (\log 2) \left(\frac{1}{q_{l}} + \frac{l \max_{0 \leq i < l}b_{i}}{q_{l}}\right) + 3 \, \frac{l  \log q_{l}}{q_{l}}.
\end{align*}
We have 
\begin{align*}
q_l  \ge b_{l-1} q_{l-1}+q_{l-2}
 \ge  b_{l-1} b_{l-2} q_{l-2}+b_{l-1} q_{l-3}+q_{l-2}
 \ge  (b_{l-1}b_{l-2}+1) q_{l-2}.
\end{align*}
By iteration we obtain
\begin{eqnarray*}
q_l \ge (b_{l-1}b_{l-2}+1)(b_{l-3}b_{l-4}+1) \cdots (b_1 b_0+1)\ge  2^{\frac{l}{2}-1}  \max_{0 \leq i < l}b_{i}
\end{eqnarray*}
if $l$ is even and
$$q_l \ge (b_{l-1}b_{l-2}+1)(b_{l-3}b_{l-4}+1) \cdots (b_2 b_1+1) q_1 \ge 2^{\frac{l-3}{2}}  \max_{0 \leq i < l}b_{i}$$
if $l$ is odd. With these estimates we get
\begin{eqnarray*}
\frac{1}{N} \sum^{N}_{n=1} \log |2 \sin (\pi n \alpha)|  \leq (\log 2) \left(\frac{1}{q_{l}} + \frac{l}{2^{(l-3)/2}}\right) + 3 \, \frac{l  \log q_{l}}{q_{l}}.
\end{eqnarray*}
Note that $q_l \ge \phi^{l-1}$ and hence $l\le \frac{\log q_l}{\log \phi}+1$, where $\phi=(1+\sqrt{5})/2$. Hence
\begin{eqnarray*}
\frac{1}{N} \sum^{N}_{n=1} \log |2 \sin (\pi n \alpha)| \leq (\log 2) \left(\frac{1}{q_{l}} + \frac{l}{2^{(l-3)/2}}\right) + 3 \, \frac{\log q_{l}}{q_{l}} \left(\frac{\log q_l}{\log \phi}+1 \right).
\end{eqnarray*}
\end{proof}

\begin{proof}[Proof of Corollary~\ref{co_b}]
Since $\alpha$ is of type $t >1$ we have
$$
\frac{c}{q_{i}^{1+t}} < \left| \alpha - \frac{p_{i}}{q_{i}}\right| < \frac{1}{a_{i+1}q_{i}^{2}}
$$
and hence $b_{i} \leq a_{i+1} < q_{i}^{t-1}/c$. Especially we have the following:
Let $b_{l} := q_{l}^{\gamma}$, then, because of
$$
q_{l}^{\gamma+1} = b_{l} q_{l} \leq N < \left(b_{l}+1\right)q_{l} \leq 2 q_{l}^{\gamma+1},
$$
we have
$$
b_{l} = q_{l}^{\gamma} \leq N^{\frac{\gamma}{\gamma+1}} \leq c_{1} N^{1-1/t}.
$$
Hence the bound from Theorem~\ref{th_prodb} can be estimated by
\begin{eqnarray*}
\prod^{l}_{i=0} 2^{b_{i}} q_{i}^{3} & \leq &  2^{b_{l}}  \left(\prod^{l}_{i=0} q_{i}^{3} \right) \prod^{l-1}_{i=0} 2^{b_{i}}\\
& \leq & 2^{c_{1} N^{1-1/t}} N^{3 \left(l+1\right)} \prod^{l-1}_{i=0} 2^{c_{1}  N^{\left(1-1/t\right)  \left(1/t\right)^{i}}} \\
& \leq & 2^{c_{2} N^{1-1/t}}  N^{c_{3}  \log N} \leq 2^{C N^{1-1/t}} 
\end{eqnarray*}
for $N$ large enough.
\end{proof}

\section{Proof of the result on the van der Corput sequence} \label{sect_d}

Let
$$
P_N:= \prod^{N}_{k=1} 2 \sin (\pi x_k)~\mbox{and}~f(k) := 2 \sin (\pi x_k),
$$
where $x_k$ is the $k^{{\rm th}}$ element of the van der Corput sequence.

\begin{lemma} \label{lem_a}
Let (in dyadic representation)
$$n:= a_{s} a_{s-1} \ldots a_{k+1}\underbrace{ 0 1 1 \ldots 1 1}_{a_ka_{k-1}\ldots a_{l+1}} \underbrace{0 1 1 \ldots 1}_{a_l a_{l-1} \ldots a_0}$$
and 
$$\overline{n} := a_{s} a_{s-1} \ldots a_{k+1} 111 \ldots 11011 \ldots1.$$
Then $P_{\overline{n}} > 2 P_{n}$.
\end{lemma}

\begin{proof}
We have 
$$
P_{\overline{n}}=P_n \, \frac{f(n+1) \cdots  f(n+2^{l}) f(n+2^{l}+1)\cdots  f(n+2^{l}+2^{k})}{f(n+2^{k}+1)\cdots f(n+2^{k}+2^{l})}.
$$
Since 
$\{x_{n+1}, \ldots, x_{n+2^{l}}\} = \{\xi, \xi + \frac{1}{2^{l}}, \ldots, \xi + \frac{2^{l}-1}{2^{l}}\}$ with
$$
\xi =\frac{1}{2^{l+1}} + \cdots + \frac{1}{2^{k}} + \frac{a_{k+1}}{2^{k+2}} + \cdots + \frac{a_{s}}{2^{s+1}},
$$
we obtain from equation (ii) of Lemma~\ref{lem_c_fritz}
$$
f(n+1) \cdots f(n+2^{l}) = 2 \sin( \pi 2^l \xi).
$$
Furthermore, 
$\{x_{n+2^{l}+1}, \ldots, x_{n+2^{l}+2^{k}}\}= \{y, y+\frac{1}{2^{k}}, \ldots, y + \frac{2^{k}-1}{2^{k}}\}$
with $$y=\frac{1}{2^{k+1}} + \frac{a_{k+1}}{2^{k+2}} + \cdots + \frac{a_{s}}{2^{s+1}}$$ and hence, again by equation (ii) of Lemma~\ref{lem_c_fritz},
$$
f(n+2^{l}+1) \cdots f(n+2^{l}+2^{k}+1) = 2 \sin (\pi 2^k y).
$$
Note that $\frac{1}{2^{k+1}} < y < \frac{1}{2^{k}}$.

In the same way we have
$\{x_{n+2^{k}+1}, \ldots, x_{n+2^{k}+2^{l}}\} = \{\tau, \tau + \frac{1}{2^{l}}, \ldots, \tau + \frac{2^{l}-1}{2^{l}}\}$
with
$$
\tau = \frac{1}{2^{l+1}} + \cdots + \frac{1}{2^{k+1}} + \frac{a_{k+1}}{2^{k+2}} + \cdots + \frac{a_{s}}{2^{s+1}}
$$
and hence by equation (ii) of Lemma~\ref{lem_c_fritz}
$$
f(n + 2^{k}+1) \cdots f(n+2^{k}+2^{l}) = 2 \sin( \pi 2^l \tau).
$$
So
$$
P_{\overline{n}}=P_n \frac{2 \sin(\pi 2^l \xi) \sin(\pi 2^k y)}{\sin(\pi 2^l \tau)}.
$$
We have to show that
$$
\Gamma:= \frac{2 \sin(\pi 2^l \xi) \sin(\pi 2^k y)}{\sin(\pi 2^l \tau)} > 1.
$$

Since $\tau = y + \frac{1}{2^{l}}-\frac{1}{2^{k}}$ and $\xi = y + \frac{1}{2^{l}}-\frac{1}{2^{k}}-\frac{1}{2^{k+1}}$ it follows that  
$$
\Gamma =  \frac{2\sin( \pi (2^{l} y+1-\frac{1}{2^{k-l}} - \frac{1}{2^{k+1-l}}))  \sin(\pi 2^{k}y)}{\sin(\pi (2^{l} y+1-\frac{1}{2^{k-l}}))}.
$$
Let $k-l =: m$ and $2^{l} y =: \eta$. Then we have  $\frac{1}{2^{m+1}}<\eta<\frac{1}{2^{m}}$ and 
$$
\Gamma = \frac{2 \sin (\pi (\eta +1-\frac{1}{2^{m}} - \frac{1}{2^{m+1}})) \sin(\pi 2^{m} \eta)}{\sin (\pi (\eta +1-\frac{1}{2^{m}}))}.
$$
Let $z:= \frac{1}{2^{m}}-\eta$. Then we have $0 < z < \frac{1}{2^{m+1}}$ and
\begin{align*}
\Gamma  = & \frac{2 \sin(\pi(1-z-\frac{1}{2^{m+1}})) \sin(\pi(1-2^{m}z))}{\sin (\pi(1-z))} \\
 = & \frac{2 \sin(\pi (z+\frac{1}{2^{m+1}})) \sin(\pi 2^m z)}{\sin (\pi z)} \\
 > &  \frac{2 \sin(\pi \frac{1}{2^{m+1}}) \sin (\pi 2^{m} \frac{1}{2^{m+1}})}{\sin(\pi \frac{1}{2^{m+1}})} \\
 = & 2.
\end{align*}
Here we used that $\sin (\pi (z + \frac{1}{2^{m+1}}))$ for $0 < z< \frac{1}{2^{m+1}}$ is minimal for $z \rightarrow 0$ and $\frac{\sin (\pi 2^{m} z)}{\sin( \pi z)}$ for $0 < z < \frac{1}{2^{m+1}}$ is minimal for $z\rightarrow \frac{1}{2^{m+1}}$.
\end{proof}

\begin{lemma} \label{lem_b}
We have:
\begin{enumerate}[(i)]
\item \begin{tabular}{llll} Let & $n$ & $=$ & $\overset{s}{\overset{\downarrow}{1}}111 \ldots 111\hspace{-0,172cm}\overset{k+1}{\overset{\downarrow}{0}}\hspace{-0,172cm}111 \ldots 1110$ \\
and & $\overline{n}$ & $=$ & $1111 \ldots 1111011 \ldots 1110$ \\
\end{tabular}

\hspace{-0,27cm}then $P_{\overline{n}} \geq P_n$.
\item \begin{tabular}{llll} Let & $n$ & $=$ & $1\hspace{-0,165cm}\overset{s-1}{\overset{\downarrow}{0}}\hspace{-0,165cm}11 \ldots111\hspace{-0,165cm}\overset{k+1}{\overset{\downarrow}{0}}\hspace{-0,165cm}111 \ldots1110$ \\
and & $\overline{n}$ & $=$ & $1011\ldots 1111011 \ldots 1110$ \\
\end{tabular}

\hspace{-0,27cm}then $P_{\overline{n}} \geq P_n$.
\item \begin{tabular}{llll} Let & $n$ & $=$ & $1111\ldots111\hspace{-0,172cm}\overset{k+1}{\overset{\downarrow}{0}}\hspace{-0,172cm}111 \ldots1111$ \\
and & $\overline{n}$ & $=$ & $1111\ldots1111011\ldots1111$ \\
\end{tabular}

\hspace{-0,27cm}then $P_{\overline{n}} \geq P_n.$\\

\item \begin{tabular}{llll} Let & $n$ & $=$ & $1011 \ldots1110111 \ldots1111$ \\
and & $\overline{n}$ & $=$ & $1011\ldots1111011\ldots1111$ \\
\end{tabular}

\hspace{-0,27cm}then $P_{\overline{n}} \geq P_n$.
\end{enumerate}
\end{lemma}

\begin{proof}
We only prove (ii), which is the most elaborate part of the lemma. The other assertions can be handled in the same way but even simpler. In (ii) we have
\begin{align*}
P_{\overline{n}} = & P_n  f(10111\ldots110111 \ldots111) \prod^{2^{k}-2}_{i=0} f(1011\ldots100\ldots0+i) \\
= & P_n 2 \sin \left(\pi \left(1-\frac{1}{2^{k+2}}- \frac{3}{2^{s+1}}\right)\right)  \frac{\sin (\pi x)}{\sin(\pi \frac{1-x}{2^{k}})}
\end{align*}
with $x=2^k(\frac{1}{2^{k+1}} - \frac{3}{2^{s+1}})$. Hence
$$
P_{\overline{n}} = P_n \frac{2 \sin (\pi (\frac{1}{2^{k+2}} + \frac{3}{2^{s+1}}) \cos(\pi \frac{3}{2^{s-k+1}})}{\sin(\pi (\frac{1}{2^{k+1}} + \frac{3}{2^{s+1}}))}.
$$
Here $s \geq 4$ and $1 \leq k\leq s-3$. Some tedious but elementary analysis of the function
$$
g(x,y) := \frac{2\sin(\pi(\frac{x}{4} + \frac{3}{2} y)) \cos(\pi \frac{3}{2} \frac{y}{x})}{\sin (\pi (\frac{x}{2} + \frac{3}{2} y))}
$$
for $0 < y \leq \frac{1}{16}$ and $8 y \leq x\leq \frac{1}{2}$ shows that $g (x,y) > 1$ in this region. Hence $P_{\overline{n}} > P_n$.
\end{proof}

\begin{proof}[Proof of Theorem~\ref{th_a}]
Consider $n$ with $2^{s} \leq n < 2^{s+1}$. From Lemma~\ref{lem_a} and Lemma~\ref{lem_b} it follows that for $2^{s}+2^{s-1} \leq n < 2^{s+1}$ the product $P_n$ has its largest values for
$$
n_{1}=111\ldots11110=2^{s+1}-2
$$
$$
n_{2}=111\ldots11101=2^{s+1}-3
$$
$$
n_{3}=111 \ldots11100=2^{s+1}-4
$$
and for $2^{s} \leq n < 2^{s}+2^{s-1}$ the product $P_n$ has its largest values for
$$
n_{4}=101\ldots11110=2^{s+1}-2^{s-1}-2
$$
$$
n_{5}=101 \ldots11101=2^{s+1}-2^{s-1}-3
$$
$$
n_{6}=101\ldots11100=2^{s+1}-2^{s-1}-4.
$$
By equation (i) of Lemma~\ref{lem_c_fritz} we have
$$
P_{n_{1}} = \frac{2^s}{\sin(\pi/2^{s+1})}
$$
hence $\frac{1}{n_1^2} P_{n_{1}} \rightarrow \frac{1}{2 \pi}$ for $s$ to infinity. Furthermore,
\begin{align*}
P_{n_2}  = & \frac{2^s}{\sin(\pi/2^{s+1})  f(2^{s+1}-2)} = \frac{2^s}{\sin(\pi/2^{s+1}) 2 \sin(\pi (\frac{1}{2} - \frac{1}{2^{s+1}}))} \\
= & \frac{2^{s-1}}{\sin(\pi/2^{s+1}) \cos(\pi/2^{s+1})},
\end{align*}
and hence $\frac{1}{n_2^2} P_{n_{2}} \rightarrow \frac{1}{4 \pi}$ for $s$ to infinity. Finally
\begin{align*}
P_{n_3} = & \frac{2^{s-1}}{\sin(\pi/2^{s+1}) \cos(\pi/2^{s+1}) f(2^{s+1}-3)} \\
= & \frac{2^{s-1}}{\sin(\pi/2^{s+1}) \cos(\pi/2^{s+1})2\sin(\pi (1-\frac{1}{4}-\frac{1}{2^{s+1}}))} \\
= & \frac{2^{s-2}}{\sin(\pi/2^{s+1}) \cos(\pi/2^{s+1}) \sin(\pi (\frac{1}{4} + \frac{1}{2^{s+1}}))}.
\end{align*}
Let now $2^{s}+2^{s-1}\leq n \leq n_{3}$ be arbitrary. Then
$$
\frac{1}{n^{2}} P_n \leq \frac{1}{(2^{s}+2^{s-1})^2} P_{n_3},
$$
and the last term tends to 
$$
\frac{2}{9 \pi  \sin \frac{\pi}{4}} < \frac{1}{2 \pi}.
$$
Hence for all $s$ large enough we have $\frac{1}{n^{2}} P_n < \frac{1}{2 \pi}$ for all $2^{s} + 2^{s-1} \leq n < n_{3}$. 

We still have to consider $n$ with $2^{s} \leq n < 2^{s}+2^{s-1}$. With equation (ii) of Lemma~\ref{lem_c_fritz} we have
\begin{align*}
P_{n_4} = & P_{n_1} \frac{1}{f(1011\ldots111) \prod^{2^{s-1}-2}_{i=0} f(11000\ldots00+i)} \\
= & P_{n_1} \frac{1}{2 \sin( \frac{3\pi}{2} \frac{1}{2^{s}})} \frac{\sin(\frac{\pi}{2^{s+1}})}{\sin \frac{3 \pi}{4}}.
\end{align*}
The product $\kappa_{s}$ of the last two factors tends to $\frac{1}{3 \sqrt{2}}$ for $s$ to infinity.

Furthermore, it is easily checked that $P_{n_5}$ and $P_{n_6}$ are smaller than $P_{n_4}$. Hence for all $n$ with $2^s \leq n < 2^s + 2^{s-1}$ we have
$$
\frac{P_n}{n^2} \leq \frac{P_{n_4}}{2^{2s}} = \frac{P_{n_1}}{n_1^2} \frac{(2^{s+1}-2)^2}{2^{2s}} \kappa_s
$$
which tends to $\frac{1}{2 \pi} \frac{4}{3 \sqrt{2}} < \frac{1}{2 \pi}$ for $s$ to infinity. So altogether we have shown that
$$
\limsup_{n \rightarrow \infty} \frac{1}{n^{2}} \prod^{n}_{i=1} 2 \sin(\pi x_i) = \frac{1}{2 \pi}.
$$
From Lemma~\ref{lem_a} and from equation (i) of Lemma~\ref{lem_c_fritz} it also follows that for all $s$ we have
$$
\min_{2^s \leq n < 2^{s+1}} P_{n} = P_{2^s} = 2^{s+1} \sin\left(\frac{\pi}{2^{s+1}}\right)
$$
which tends to $\pi$ for $s$ to infinity. This gives the lower bound in Theorem~\ref{th_a}.
\end{proof}

\section{Proof of the probabilistic results} \label{sect_prob}

In the first part of this section we consider products

\begin{align} P_{N}=\prod^{N}_{k=1}2\sin (\pi X_{k}), \label{8}
\end{align}

where $(X_{k})_{k\geq 1}$ is a sequence of i.i.d.\ random variables
on $[0,1]$. We want to determine the almost sure asymptotic behavior of \eqref{8}. We take logarithms and define

\begin{align}S_{N}=\log P_{N}=\sum^{N}_{k=1}\log (2\sin (\pi
X_{k}))= \sum^{N}_{k=1}Y_{k}, \label{9}
\end {align}

where $Y_{k}= \log (2\sin (\pi X_{k}))$ is again an i.i.d.\ sequence.
Thus we can apply Kolmogorov's law of the iterated logarithm \cite{kolm} (see also Feller \cite{feller})
in the i.i.d.\ case. However, for later use we state this LIL in a more general form below.

\begin{lemma} \label{lemma7}
Let $(Z_{k})_{k\geq 1}$ be
a sequence of independent random variables with expectations $\mathbb{E}Z_{k}= 0$ and finite variances $\mathbb{E}Z^{2}_{k}<\infty$, and let
$B_{N}=\sum^{N}_{k=1}\mathbb{E}Z_{k}^{2}$. Assume there are
positive numbers $M_{N}$ such that
\begin{align}| Z_{N}|\leq M_{N} \qquad \text{and} \qquad M_{N}=o\left(\sqrt{\frac{B_{N}}{\log\log
B_{N}}}\right). \label{10}
\end{align}
Then $S_{N}=\sum^{N}_{k=1}Z_{k}$ satisfies a law of the iterated
logarithm
\begin{align}\limsup_{N\rightarrow\infty}\frac{S_{N}}{\sqrt{B_{N}\log\log
B_{N}}}= \sqrt{2} \qquad \text{almost surely.}
\end{align}
\end{lemma}

In the case of centered i.i.d random variables
${Z}_{k}$ with finite variance, we have $B_{N}=bN$ with
$b=\mathbb{E}Z^{2}_{1}$. Thus in this case

\begin{align}\limsup_{N\rightarrow\infty}\frac{S_{N}}{\sqrt{N \log\log
N}}=\sqrt{2 b}\quad\quad\quad\text{almost surely.}
\end{align}

In order to apply Lemma~\ref{lemma7} to the sum \eqref{9}, we note that
$$\mathbb{E} Y_k = \mathbb{E} (\log(2\sin (\pi X_{k}))) = \int^{1}_{0}\log (2 \sin (\pi x)) \rd x= 0,$$
and compute the variance
$$\mathbb{E} Y_k^ 2 = \mathbb{E} (\log^{2}(2\sin (\pi X_{k}))) =\int^{1}_{0}\log^{2}(2\sin(\pi x)) \rd x=\frac{\pi^2}{12}.$$
This proves Theorem \ref{th51}.\\

For the proof of Theorem \ref{th52} we split
the corresponding logarithmic sum into two parts
\begin{eqnarray}
\lefteqn{\sum_{1 \leq n_k \leq N} \log (2\sin (\pi n_{k}\alpha))} \nonumber\\
& = & \frac{1}{2} \left(\sum^{N}_{n=1}\log (2\sin (\pi
n \alpha)) + \sum^{N}_{n=1} R_{n}\log (2\sin (\pi n\alpha)) \right), \label{13}
\end{eqnarray}
where $R_{n}=R_{n}(t)$ denotes the $n^{{\rm th}}$ Rademacher function on
$[0,1]$ and the space of subsequences of the positive integers
corresponds to $[0,1]$ equipped with the Lebesgue measure. For irrationals $\alpha$ with bounded continued
fraction expansion, by Corollary~\ref{co_prodb} we have

\begin{align}\sum^{N}_{n=1}\log (2\sin (\pi
n \alpha)) =O(\log^{2}N). \label{14}
\end{align}

For the second sum in \eqref{13} we set $Z_{n}=R_{n}\log (2 \sin (\pi
n \alpha))$ and apply Lemma~\ref{lemma7}. The random variables $Z_{n}$ are
clearly independent and thus we have to compute the quantities
$B_{N}$ and check condition \eqref{10}. Obviously, $\mathbb{E}Z_{n}=0$
and $\mathbb{E}Z^{2}_{n}=\log^{2} (2\sin (\pi n \alpha))$.
Using the fact that
$$
|\sin (\pi n \alpha)|\geq 2 \| n \alpha \| \geq\frac{{c}_{0}}{n},
$$
with some positive constant $c_0$, we obtain
$$| Z_{N}| \leq c_{1} \log N$$

with some $c_{1}> 0$. Using Koksma's inequality and discrepancy estimates for $(n \alpha)_{n \geq 1}$ it can easily been shown that
\begin{eqnarray*}
\frac{B_{N}}{N} = \frac{1}{N}\sum_{n=1}^N \log^{2} (2\sin (\pi n \alpha))
 \to  \int_0^1 \log^{2} (2\sin (\pi n \alpha)) \rd \alpha = \frac{\pi^2}{12}.
\end{eqnarray*}
Thus, the conditions of Lemma~\ref{lemma7} are satisfied and we have
$$
\limsup_{N \to \infty} \frac{\sum_{n=1}^N Y_n}{\sqrt{N \log \log N}} = \frac{\pi}{\sqrt{6}},\qquad \text{$\mathbb{P}$-almost surely.}
$$
Consequently, from \eqref{13} and \eqref{14} we obtain 
\begin{equation} \label{equfin}
\limsup_{N \to \infty} \frac{\sum_{1 \leq n_k \leq N} \log (2\sin (\pi n_{k}\alpha))}{\sqrt{N \log \log N}}  = \frac{\pi}{2\sqrt{6}}, \qquad \text{$\mathbb{P}$-almost surely.}
\end{equation}
Finally, note that by the strong law of large numbers we have, $\mathbb{P}$-almost surely, that
$$
\# \left\{k:~1 \leq n_k \leq N\right\} \sim \frac{N}{2}.
$$
Consequently, from \eqref{equfin} we can deduce that
$$
\limsup_{N \to \infty} \frac{\sum_{k=1}^N \log (2\sin (\pi n_{k}\alpha))}{\sqrt{N \log \log N}}  = \frac{\pi}{\sqrt{12}}, \qquad \text{$\mathbb{P}$-almost surely.}
$$
This proves Theorem \ref{th52}.\\

\textbf{Acknowledgment.}
We thank Dmitriy Bilyk who drew our attention to the results given in \cite{Knill+Les, Lub}, and \cite{VerMes}.

\end{document}